\documentclass[11pt,reqno]{amsart}

\usepackage{mathtools}
\usepackage{amssymb,amsthm,amscd,array}
\usepackage{graphics}
\usepackage[final]{epsfig}
\usepackage[all]{xy}

\let\ssection=\section
\renewcommand{\section}{\setcounter{equation}{0}\ssection}

\newcommand {\emptycomment}[1]{}

\newcommand{\bbZ}{\mathbb{Z}}

\newcommand{\ad}{\mathrm{ad}}

\newcommand{\Hom}{\mathrm{Hom}}
\newcommand{\End}{\mathrm{End}}

\newcommand{\half}{\textstyle{\frac{1}{2}}}

\newcommand{\fa}{\mathfrak{a}}

\chardef\s=110
\chardef\g=103

\newtheorem{theo}{Theorem}
\newtheorem{lemma}{Lemma}[section]

\newtheorem{prop}[lemma]{Proposition}

\newtheorem{exe}[lemma]{Example}
\newtheorem{defi}[lemma]{Definition}

\def\a{\alpha}
\def\b{\beta}

\def\e{\varepsilon}

\def\g{\gamma}

\def\r{\rho}
\def\s{\sigma}

\begin{document}

\title{On Hom-Lie antialgebra}

\author[T. Zhang]{Tao Zhang}
\address{College of Mathematics and Information Science\\
Henan Normal University\\
Xinxiang 453007, PR China}
\email{zhangtao@htu.edu.cn}

\author[H. Zhang]{Heyu Zhang}
\address{College of Mathematics and Information Science\\
Henan Normal University\\
Xinxiang 453007, PR China}
\email{zhy199404@126.com}

\date{}

\begin{abstract}
In this paper, we introduced the notion of Hom-Lie antialgebras.
The representations and cohomology theory of Hom-Lie antialgebras are investigated. We prove that the equivalent classes of abelian extensions of Hom-Lie antialgebras are in one-to-one correspondence to elements of the second cohomology group. We also prove that 1-parameter infinitesimal deformation of a Hom-Lie antialgebra are characterized by 2-cocycles of this Hom-Lie antialgebra with adjoint representation in itself. The notion of Nijenhuis operators of Hom-Lie antialgebra is introduced to describe trivial deformations.
\end{abstract}



\maketitle

\thispagestyle{empty}


\section{Introduction}
The notion of Lie antialgebras was introduced by Ovsienko in \cite{Ovs}. A Lie antialgebra is a $\bbZ_2$-graded vector space $\fa=\fa_0\oplus\fa_1$ where $\fa_0$ is a commutative associative algebra, $\fa_1$ is equipped with a map $\fa_1\times \fa_1 \to \fa_0$ satisfying some axioms similar to the axioms of Lie algebras, and $\fa_0$ acts commutatively on $\fa_1$ as a derivation, see Definition \ref{def001}.
The representations, universal enveloping algebra and cohomology theory of Lie antialgebras have been investigated in \cite{LM}, \cite{MG} and \cite{LO}.



On the other hand, Hom-type algebras was introduced to deal with $q$-deformations of algebras of vector fields \cite{HLS}.
A Hom-associative algebra is a vector space $A$ with an additional linear map $\a:A \to A$  satisfying the Hom-associative identity: $$\a(x_1)\cdot(x_2\cdot x_3)=(x_1\cdot x_2)\cdot\a(x_3),$$
for all $x_1,x_2,x_3\in A$.
A Hom-Lie algebra is a vector space $\mathfrak{g}$ with an additional linear map $\a:\mathfrak{g}
\to \mathfrak{g}$  satisfying the Hom-Jacobi identity:
$$[\a(y_1),[y_2,y_3]]+[\a(y_2),[y_3,y_1]]+[a(y_3),[y_1,y_2]]=0,$$
for all $y_1,y_2,y_3\in \mathfrak{g}$.
The representations, abelian extensions, deformations  and cohomology theory of Hom-algebras were studied in \cite{AEM,MS,SY}. Universal central extensions of Hom-Lie algebras were studied in \cite{CIP,LS}. It is known that abelian extensions and deformations of Hom-type algebras are governed by the second cohomology group.
Other types of Hom-structures include  BiHom-Lie algebras, Hom-Nambu-Lie algebras, Hom-Hopf algebras, Hom-Poisson algebras and
Hom-Lie-Yamaguti algebras, see \cite{GAC,Ma,Yau1,Yau2,Zhang}.

Motivated by the above results, we introduce the notion of a Hom-Lie antialgebra in this paper.
We define a Hom-Lie antialgebra $(\fa,\a,\b)$ as a supercommutative $\bbZ_2$-graded algebra $\fa=\fa_0\oplus\fa_1$  with
two linear maps $\alpha:\fa_0\rightarrow \fa_0,~\beta:\fa_1\rightarrow \fa_1$ satisfying some compatibility conditions, see Definition \ref{def1}.
When both $\alpha $ and $\beta$ are identity maps, we get the ordinary notion of a Lie antialgebra.
Note that $\alpha $ and $\beta$ are not equal to each other since they act on different spaces.
This is the key difference of Hom-Lie antialgebra in this paper and other types of Hom-algebras in the literature.
The representations and cohomology groups of Hom-Lie antialgebras are investigated.
We also study deformations and abelian extensions of Hom-Lie antialgebras which are described by the second cohomology group.
For the third cohomology group, it will be  related to  the crossed module extensions of a Hom-Lie antialgebra. This is investigated in our subsequent paper \cite{ZZ2}.

The paper is organized as follows. In Section 2, we give some  definitions and notations of Hom-Lie antialgebras. In Section 3, we study representations of Hom-Lie antialgebras and define the cohomology groups of Hom-Lie antialgebras. In Section 4, we study abelian extensions of Hom-Lie antialgebras using the cohomology theory defined in Section 3 and prove that abelian extensions are classified by the second cohomology group. In Section 5, we study 1-parameter infinitesimal deformations of a Hom-Lie antialgebra. The notion of a Nijenhuis operator on a Hom-Lie antialgebra is introduced to describe trivial deformations.

Throughout this paper, we work with an algebraically closed field  of characteristic 0.
For a  $\bbZ_2$-graded vector space $V=V_0\oplus V_1$ , we consider the standard  $\bbZ_2$-grading on the algebra of linear maps on  $V$: $\End(V)=\End(V)_0\oplus \End(V)_1$ where $\End(V)_0=\Hom(V_0,V_0)\oplus\Hom(V_1,V_1)$ and $\End(V)_1=\Hom(V_0,V_1)\oplus\Hom(V_1,V_0)$.
\section{Hom-Lie antialgebras}\label{defintion}
\begin{defi}\label{def001}
A Lie antialgebra is a supercommutative $\bbZ_2$-graded algebra:
$\fa=\fa_0\oplus\fa_1$,
$\fa_i\cdot\fa_j\subset\fa_{i+j},$
such that  the following identities hold:
\begin{eqnarray}
\label{AssCommT0}
x_1\cdot\left(x_2\cdot x_3\right)&=&\left(x_1\cdot x_2\right)\cdot x_3,\\
\label{CacT0}
x_1\cdot(x_2\cdot y_1)&=&\half(x_1\cdot x_2)\cdot y_1,\\
\label{ICommT0}
x_1\cdot[y_1,y_2]&=&[x_1\cdot y_1,\,  y_2]\;+\;[y_1,\, x_1\cdot y_2],\\
\label{Jack0}
y_1\cdot [y_2,y_3]&+&y_2 \cdot[y_3,y_1]\;+\;y_3\cdot [y_1,y_2]=0,
\end{eqnarray}
for all homogeneous elements $x_1, x_2, x_3\in\fa_0,~ y_1, y_2, y_3\in\fa_1$.
Note that we denote $y_1\cdot y_2$ by $[y_1,y_2]$ in the last equation \eqref{Jack0} since it is anti-commutative,  which is slightly different from notations in \cite{Ovs,LO}.
\end{defi}


\begin{defi}
Let $\fa$ and $\widetilde{\fa}$ be Lie antialgebras. An algebraic homomorphism $\phi$ from $\fa$ to $\widetilde{\fa}$ consists of  $\phi_0:\fa_0\rightarrow\widetilde{\fa}_0 $, $\phi_1:\fa_1\rightarrow\widetilde{\fa}_1 $, such that the following conditions hold:
\begin{eqnarray}
\label{Thomo1}
\phi_0(x_1\cdot x_2)&=&\phi_0(x_1)\cdot\phi_0(x_2),\\
\label{Thomo2}
\phi_1(x_1\cdot y_1)&=&\phi_0(x_1)\cdot\phi_1(y_1),\\
\label{Thomo3}
\phi_0([y_1,y_2])&=&[\phi_1(y_1),\phi_1(y_2)],
\end{eqnarray}
for all $x_1, x_2, x_3\in\fa_0,~ y_1, y_2, y_3\in\fa_1$.
\end{defi}

\begin{defi}\label{def1}
A Hom-Lie antialgebra $(\fa,\a,\b)$ is a supercommutative $\bbZ_2$-graded algebra $\fa=\fa_0\oplus\fa_1$, together with two linear maps $\alpha:\fa_0\rightarrow \fa_0,~\beta:\fa_1\rightarrow \fa_1$,
satisfying the following identities:
\begin{eqnarray}
\label{hanti01}
\a(x_1)\cdot\left(x_2\cdot x_3\right)&=&\left(x_1\cdot x_2\right)\cdot\a(x_3),\\
\label{hanti02}
\a(x_1)\cdot(x_2\cdot y_1)&=&\half(x_1\cdot x_2)\cdot\b(y_1),\\
\label{hanti03}
\a(x_1)\cdot[y_1,y_2]&=&[x_1\cdot y_1,\b(y_2)]\;+\;[\b(y_1),x_1\cdot y_2],\\
\label{hanti04}
\b(y_1)\cdot[y_2,y_3]&+&\b(y_2)\cdot[y_3,y_1]\;+\;\b(y_3)\cdot[y_1,y_2]=0,
\end{eqnarray}
for all $x_1, x_2, x_3\in\fa_0,~ y_1, y_2, y_3\in\fa_1$. We also call such systems $(\a,\b)$-Hom-Lie antialgebra.
\end{defi}

We give some explanations about the meaning of equalities \eqref{hanti01}--\eqref{hanti04}.
From equality \eqref{hanti01}, $\fa_0$ is a Hom-associative subalgebra of $\fa$.
The equality \eqref{hanti02} and \eqref{hanti03} mean that $2\ad_x: x\to 2x\cdot y$ is an action of $\fa_0$ on $\fa_1$ as derivations.
A Hom-Lie antialgebra is called abelian if the product $x_1\cdot x_2, x_1\cdot y_1$ and bracket $[y_1, y_2]$ are all zero for any $x_i\in \fa_0,y_i\in \fa_1$.

A Hom-Lie antialgebra is called multiplicative if $(\a,\b)$ form an algebraic homomorphism of $\fa$, i.e. for any $x_1, x_2\in\fa_0,~ y_1, y_2\in\fa_1$, we have
\begin{eqnarray}
\a(x_1\cdot x_2)&=&\a(x_1)\cdot\a(x_2),\\
\b(x_1\cdot y_1)&=&\a(x_1)\cdot\b(y_1),\\
\a([y_1,y_2])&=&[\b(y_1),\b(y_2)].
\end{eqnarray}
The Hom-Lie antialgebras in this paper are assumed to be multiplicative unless otherwise stated.

\begin{defi}
Let $(\fa,\a,\b)$ and $(\widetilde{\fa},\widetilde{\a},\widetilde{\b})$ be Hom-Lie antialgebras. A Hom-Lie antialgebra homomorphism $\phi$ from $\fa$ to $\widetilde{\fa}$ consists of  $\phi_0:\fa_0\rightarrow\widetilde{\fa}_0 $, $\phi_1:\fa_1\rightarrow\widetilde{\fa}_1 $, such that the following equalities hold for all $x_1, x_2\in\fa_0,~ y_1, y_2\in\fa_1$:
\begin{eqnarray}\label{Thomo4}
\phi_0\circ\a&=&\widetilde{\a}\circ\phi_0,\\
\label{Thomo5}
\phi_1\circ\b&=&\widetilde{\b}\circ\phi_1,\\
\label{Thomo1}
\phi_0(x_1\cdot x_2)&=&\phi_0(x_1)\cdot\phi_0(x_2),\\
\label{Thomo2}
\phi_1(x_1\cdot y_1)&=&\phi_0(x_1)\cdot\phi_1(y_1),\\
\label{Thomo3}
\phi_0([y_1,y_2])&=&[\phi_1(y_1),\phi_1(y_2)].
\end{eqnarray}
\end{defi}

\begin{prop}\label{prop-1}
Let $\fa=\fa_0\oplus\fa_1$ be a Lie antialgebra and $(\alpha,\beta)$ be an algebraic homomorphism from $\fa$ to itself. Then the induced Hom-Lie antialgebra $(\fa;\cdot_{(\alpha,\beta)},[\cdot,\cdot]_\alpha,\alpha,\beta)$ is the space $\fa=\fa_0\oplus\fa_1$ under the following operations:
$$x_1\cdot_\alpha x_2=\alpha(x_1\cdot x_2),\quad x_1\cdot_\beta y_1=\beta(x_1\cdot y_1),\quad [y_1,y_2]_\alpha=\alpha([y_1,y_2]).$$
\end{prop}
\begin{proof}
Here we verify that \eqref{hanti02} and \eqref{hanti03} hold, the other two are similar.

For \eqref{hanti02},
\begin{eqnarray*}
\a(x_1)\cdot_{\b}(x_2\cdot_{\b}y_1)=\b(\a(x_1)\cdot\b(x_2\cdot y_1))=\a^2(x_1)\cdot(\a^2(x_2)\cdot\b^2(y_1)),
\end{eqnarray*}
\begin{eqnarray*}
\half(x_1\cdot_{\a}x_2)\cdot_{\b}\b(y_1)=\half\b(\a(x_1\cdot x_2)\cdot\b(y_1))=\half(\a^2(x_1)\cdot\a^2(x_2))\cdot\b^2(y_1),
\end{eqnarray*}
thus, we have
$$\a(x_1)\cdot_{\b}(x_2\cdot_{\b}y_1)=\half(x_1\cdot_{\a}x_2)\cdot_{\b}\b(y_1).$$
For \eqref{hanti03},
\begin{eqnarray*}
\a(x_1)\cdot_{\a}[y_1,y_2]_{\a}&=&\a(\a(x_1)\cdot\a[y_1,y_2])=\a^2(x_1)\cdot[\b^2(y_1),\b^2(y_2)],\\
{[(x_1\cdot_{\b}y_1),\b(y_2)]_{\a}}&=&\a[\b(x_1\cdot y_1),\b(y_2)]=[\a^2(x_1)\cdot\b^2(y_1),\b^2(y_2)],\\
{[\b(y_1),x_1\cdot_\b y_2]_\a}&=&[\b^2(y_1),\b^2(x_1\cdot y_2)]=[\b^2(y_1),\a^2(x_1)\cdot\b^2(y_2)],
\end{eqnarray*}
thus, we obtain
$$\a(x_1)\cdot_{\a}[y_1,y_2]_{\a}=[(x_1\cdot_{\b}y_1),\b(y_2)]_{\a}+[\b(y_1),x_1\cdot_\b y_2]_\a.$$
\end{proof}

By the above Proposition \ref{prop-1}, we can construct examples of Hom-Lie antialgebras as follows.
More general constructions are given in the next sections.
\begin{exe}
{\rm Consider the Lie antialgebra $K(1)$ introduced in \cite{Ovs} as follows.
This algebra has the basis $\{\e, a,b\}$,
where $\e$ is even and $a,b$ are odd,
satisfying the relations
\begin{eqnarray*}
\label{aslA}
\e\cdot{}\e=\e,\quad\e\cdot{}a=\half\,a,\quad\e\cdot{}b=\half\,b,\quad{[a,b]}=\half\,\e.
\end{eqnarray*}
Consider the linear map $(\a,\b):\fa\to \fa$ defined by
$$\a(\e)=\e, \quad \b(a)=\mu a,\quad \b(b)=\mu^{-1} b$$
on the basis elements. This map is actually a Lie antialgebra homomorphism.
By Proposition \ref{prop-1}, we obtain a Hom-Lie algebra structure given by
\begin{eqnarray*}
\label{aslA}
\e\cdot{}\e=\e,\quad
\e\cdot{}a=\half\,\mu\, a,\quad
\e\cdot{}b=\textstyle{\frac{1}{2}}\mu^{-1} b,\quad
{[a,b]}=\half\,\e.
\end{eqnarray*}
}
\end{exe}

\begin{exe}
\label{ExMain}
{\rm
Another example of a Hom-Lie antialgebra is the
\textit{conformal Hom-Lie antialgebra}.
This is a simple infinite-dimensional Hom-Lie antialgebra with the basis
$$
\textstyle
\left\{
\e_n,\;n\in\bbZ;
~a_i;~i\in\bbZ+\half
\right\},
$$
where $\e_n$ are even, $a_i$ are odd, and
$\a(\e_i)=\e_i$, $\b(a_i)=(1+q^i) a_i$
satisfy the following relations:
\begin{eqnarray*}
\label{GhosRel}
\e_n\cdot{}\e_m&=&\e_{n+m},\\
\e_n\cdot{}a_i &=&\half (1+q^i) a_{n+i},\\
{[a_i, a_j]}&=&\half\left(\{j\}-\{i\}\right)\e_{i+j},
\end{eqnarray*}
where $\{i\}=(q^i-1)/(q-1),\  q\neq 1$.
}
\end{exe}

\section{Representations and cohomology}\label{Rep}
In this section, we introduce the notion of a representation for the class of Hom-Lie antialgebras.
Then we study the semidirect products and cohomology groups  of Hom-Lie antialgebras.
\begin{defi}
Let $(\fa,\a,\b)$ be a Hom-Lie antialgebra, $V=V_0\oplus V_1$  be a Hom-super vector space (a super vector space with linear maps $\a_{V_0}\in\Hom(V_0,V_0),~\b_{V_1}\in\Hom(V_1,V_1)$). A representation of $(\fa,\a,\b)$ over the Hom-super vector space $V$ is a pair of linear maps $\r=(\r_0,\r_1):\r_0:\fa_0\rightarrow \End(V)_0,~\r_1:\fa_1\rightarrow \End(V)_1$ such that the following conditions hold:
\begin{eqnarray}
\label{rep01}
\r_0(\a(x_1))\circ\r_0(x_2)(u_1)&=&\r_0(x_1\cdot x_2)\circ\a_{V_0}(u_1),\\
\label{rep02}
\r_0(\a(x_1))\circ\r_0(x_2)(w_1)&=&\half\r_0(x_1\cdot x_2)\circ\b_{V_1}(w_1),\\
\label{rep03}
\r_0(\a(x_1))\circ\r_1(y_1)(u_1)&=&\half\r_1(\b(y_1))\circ\r_0(x_1)(u_1),\\
\label{rep04}
\r_1(x_1\cdot y_1)\circ\a_{V_0}(u_1)&=&\half\r_1(\b(y_1))\circ\r_0(x_1)(u_1),\\
\label{rep05}
\r_0(\a(x_1))\circ\r_1(y_1)(w_1)&=&\r_1(x_1\cdot y_1)\circ\b_{V_1}(w_1)+\r_1(\b(y_1))\circ\r_0(x_1)(w_1),\\
\label{rep06}
\r_0([y_1,y_2])\circ\a_{V_0}(u_1)&=&\r_1(\b(y_1))\circ\r_1(y_2)(u_1)-\r_1(\b(y_2))\circ\r_1(y_1)(u_1),\\
\label{rep07}
\r_0([y_1,y_2])\circ\b_{V_1}(w_1)&=&\r_1(\b(y_2))\circ\r_1(y_1)(w_1)-\r_1(\b(y_1))\circ\r_1(y_2)(w_1),
\end{eqnarray}
for all $x_1,x_2\in\fa_0,~y_1,y_2\in\fa_1$, $u_1\in V_0,~w_1\in V_1$.
\end{defi}
The above conditions seem very complicated at first glance. We give an equivalent condition as follows.


\begin{prop}\label{prop-rep}
Let $(\fa,\a,\b)$ be a Hom-Lie antialgebra, $(V,\a_{V_0},\b_{V_1})$ be a Hom-super vector space. Then
$\r=(\r_0,\r_1)$  is a representation of $(\fa,\a,\b)$ over $(V,\a_{V_0},\b_{V_1})$  if and only if $\fa\oplus V\triangleq(\fa_0\oplus V_0)\oplus(\fa_1\oplus V_1)$ is a Hom-Lie antialgebra under the following operations:
\begin{eqnarray}\label{homo10}
(\a+\a_{V_0})(x_1,u_1)&=&(\a(x_1),\a_{V_0}(u_1)),\\
\label{homo20}
(\b+\b_{V_1})(y_1,w_1)&=&(\b(y_1),\b_{V_1}(w_1)),\\
\label{op10}
(x_1,u_1)\cdot(x_2,u_2)&=&(x_1\cdot x_2,~ \r_0(x_1)(u_2)+\r_0(x_2)(u_1)),\\
\label{op20}
(x_1,u_1)\cdot(y_1,w_1)&=&(x_1\cdot y_1,~ \r_0(x_1)(w_1)+\r_1(y_1)(u_1)),\\
\label{op30}
[(y_1,w_1),(y_2,w_2)]&=&([y_1,y_2], \r_1(y_1)(w_2)-\r_1(y_2)(w_1)),
\end{eqnarray}
for all $x_1,x_2\in\fa_0,~y_1,y_2\in\fa_1$, $u_1,u_2\in V_0,~w_1,w_2\in V_1$.
This is called a semidirect product of $\fa$ and $V$,
denoted by $\fa\ltimes V$.
\end{prop}
The proof of the above proposition \ref{prop-rep} is by direct computations, so we omit the details.

Now we define the generalized Chevalley-Eilenberg complex for Hom-Lie antialgebra $(\fa,\a,\b)$ with coefficients in $V$.
Given a Hom-Lie antialgebra $\fa$ with representation in $V$, we define $C^{m,n}(\fa,V)$ to be the space of multi-linear maps
$$f:(\fa_0\otimes\cdots\otimes \fa_0)\otimes (\fa_1\wedge\cdots\wedge \fa_1)\to V$$
with $x_1,\cdots,x_m\in \fa_0$ and $y_1,\cdots,y_n\in \fa_1$
such that 
\begin{eqnarray}\label{cochain1}
(\a_{V_0},\b_{V_1})f(x_1,\cdots,x_m,y_1,\cdots,y_n)=f(\a(x_1),\cdots,\a(x_m),\b(y_1),\cdots,\b(y_n)).
\end{eqnarray}
Denote by $C^k(\fa,V)$ the set of $k$-cochains:
\begin{eqnarray*}
C^k(\fa,V)&\triangleq&\bigoplus_{m+n=k}C^{m,n}(\fa,V).
\end{eqnarray*}
Define the coboundary operator $d^k=d_{1,0}^k+d_{0,1}^k+d_{-1,2}^k:C^k(\fa,V)\longrightarrow C^{k+1}(\fa,V)$, where $d_{i,j}^k:C^{m,n}(\fa,V)\to C^{m+i,n+j}(\fa,V)$ for $m+n=k$. The operator $d_{1,0}^k, d_{0,1}^k$ and $d_{-1,2}^k$ are given explicitly as follows.

\noindent
(i)
If $n=0$, $d_{1,0}^k$ is given by
\begin{eqnarray*}
&&(d^k_{1,0}f)(x_1,\ldots,x_{m+1})\\
&=&\half\{\r_0(\a^{k-1}(x_1))f(x_2,\ldots,x_{m+1})\\
&&+\sum_{i=1}^{m}
(-1)^{i+1}\,f(\a(x_1),\ldots,\a(x_{i-1}),x_i\cdot x_{i+1},\a(x_{i+2}),\ldots,\a(x_{m+1}))\\
&&+{(-1)^{m+1}}\r_0(\a^{k-1}(x_{m+1}))f(x_1,\ldots,x_{m})\};
\end{eqnarray*}
if $n>0$, and $f(x_2,\ldots,x_{m+1},\,y_1,\ldots,y_{n})$ is with values in $V_0$, then $d_{1,0}^k$ is given by
\begin{eqnarray*}
&&(d^k_{1,0}f)(x_1,\ldots,x_{m+1},\,y_1,\ldots,y_{n})\\
&=&\half\r_0(\a^{k-1}(x_1))f(x_2,\ldots,x_{m+1},\,y_1,\ldots,y_{n})\\
&&+\half\sum_{i=1}^{m}
(-1)^{i+1}\,f(\a(x_1),\ldots,\a(x_{i-1})\,, x_i\cdot x_{i+1}\,,\a(x_{i+2}),\ldots,\a(x_{m+1}),\\
&&\quad\b(y_1),\ldots,\b(y_{n}))\\
&&+\frac{1}{n}
\sum_{j=1}^{n}(-1)^{m+j+1}\,
f(\a(x_1),\ldots,\a(x_{m}), x_{m+1}\cdot y_j,\b(y_1),\ldots,\widehat{y_j},\ldots,\b(y_{n}));
\end{eqnarray*}
if $n>0$, and $f(x_2,\ldots,x_{m+1},\,y_1,\ldots,y_{n})$ is with values in $V_1$, then $d_{1,0}^k$ is given by
\begin{eqnarray*}
&&(d^k_{1,0}f)(x_1,\ldots,x_{m+1},\,y_1,\ldots,y_{n})\\
&=&\r_0(\a^{k-1}(x_1))f(x_2,\ldots,x_{m+1},\,y_1,\ldots,y_{n})\\
&&+\half\sum_{i=1}^{m}
(-1)^{i+1}\,f(\a(x_1),\ldots,\a(x_{i-1}),\,x_i\cdot x_{i+1},\a(x_{i+2}),\ldots,\a(x_{m+1}),\\
&&\quad\b(y_1),\ldots,\b(y_{n}))\\
&&+\frac{1}{n}
\sum_{j=1}^{n}(-1)^{m+j+1}\,
f(\a(x_1),\ldots,\a(x_{m}), x_{m+1}\cdot y_j,\,\b(y_1),\ldots,\widehat{y_j},\ldots,\b(y_{n})).
\end{eqnarray*}

\noindent
(ii)
If $m=0$ and $n$ is odd, then $d^k_{0,1}$ is given by
\begin{eqnarray*}
\label{TheCobOp01One}
&&(d^k_{0,1}f)(y_1,\ldots,y_{n+1})\\
&=&
\frac{2}{n+1}
\sum_{j=1}^{n+1}(-1)^{j+1}\,
\r_1(\b^{k-1}(y_j))f(y_1,\ldots,\widehat{y_j},\ldots,y_{n+1})
;
\end{eqnarray*}

if $m>0$ or if $n$ is even, then $d^k_{0,1}$ is given by
\begin{eqnarray*}
&&(d^k_{0,1}f)(x_1,\ldots,x_{m},\,y_1,\ldots,y_{n+1})\\
&=&
\frac{1}{n+1}\,
\sum_{j=1}^{n+1}(-1)^{m+j+1}\,
\r_1(\b^{k-1}(y_j))f(x_1,\ldots,x_{m},y_1,\ldots,\widehat{y_j},\ldots,y_{n+1}).
\end{eqnarray*}

\noindent
(iii)
If $m>0$, then $d^k_{-1,2}$ is given by
\begin{eqnarray*}
&&(d^k_{-1,2}\,f)(x_1,\ldots,x_{m-1},\,y_1,\ldots,y_{n+2})\nonumber\\
&=&
\frac{2}{(n+1)(n+2)}
\sum_{i<j}(-1)^{m+i+j}
f(\a(x_1),\ldots,\a(x_{m-1}),\\
&&[y_i,y_j],\,\b(y_1),\ldots,\widehat{y_i},\ldots,\widehat{y_j},\ldots,\b(y_{n+2})).\nonumber
\end{eqnarray*}
It can be proved similarly as in \cite{LO} that $d^{k+1}\circ d^k=0$, so $d$ is a coboundary operator.
Thus associated to the representation $\r$, we obtain the cochain complex $\big(C^k(\fa,V),d\big)$. Denote the set of closed $k$-cochains by $Z^k(\fa,V)$ and the set of exact $k$-cochains by $B^k(\fa,V)$. Define the corresponding cohomology group by
\begin{equation*}
H^k(\fa,V)=Z^k(\fa,V)/B^k(\fa,V).
\end{equation*}

In particular, for $k=2$, we obtain the following relations by direct computations:
\begin{eqnarray}\label{2-cocycle-f1}
(d^2f)(x_1,x_2,x_3)&=&(d^2_{1,0}f)(x_1,x_2,x_3)\notag\\
&=&\half\{\r_0(\a(x_1))f(x_2,x_3)-f(x_1\cdot x_2,\a(x_3))\notag\\
&&+f(\a(x_1),x_2\cdot x_3)-\r_0(\a(x_3))f(x_1,x_2)\},
\end{eqnarray}
\begin{eqnarray}\label{2-cocycle-f2}
(d^2f)(x_1,x_2,y_1)&=&(d^2_{1,0}f)(x_1,x_2,y_1)+(d^2_{0,1}f)(x_1,x_2,y_1)\notag\\
&=&-\r_0(\a(x_1))f(x_2,y_1)+\half f(x_1\cdot x_2,\b(y_1))\notag\\
&&-f(\a(x_1),x_2\cdot y_1)+\r_1(\b(y_1))f(x_1,x_2),
\end{eqnarray}
\begin{eqnarray}\label{2-cocycle-f3}
(d^2f)(x_1,y_1,y_2)&=&(d^2_{1,0}f)(x_1,y_1,y_2)+(d^2_{0,1}f)(x_1,y_1,y_2)\notag\\
&&+(d^2_{-1,2}f)(x_1,y_1,y_2)\notag\\
&=&\textstyle{\frac{1}{2}}\{-\r_0(\a(x_1)) f(y_1,y_2)+f(x_1\cdot y_1,\b(y_2))\notag\\
&&-f(x_1\cdot y_2,\b(y_1))-\r_1(\b(y_1))f(x_1,y_2)\notag\\
&&+\r_1(\b(y_2))f(x_1,y_1)\}-f(\a(x_1),[y_1,y_2]),
\end{eqnarray}
\begin{eqnarray}\label{2-cocycle-f4}
(d^2f)(y_1,y_2,y_3)&=&(d^2_{0,1}f)(y_1,y_2,y_3)+(d^2_{-1,2}f)(y_1,y_2,y_3)\notag\\
&=&\textstyle{\frac{1}{3}}\{\r_1(\b(y_1))(f(y_2,y_3))-\r_1(\b(y_2))(f(y_1,y_3))\notag\\
&&+\r_1(\b(y_3))(f(y_1,y_2))-f([y_1,y_2],\b(y_3))\notag\\
&&+f([y_1,y_3],\b(y_2))-f([y_2,y_3],\b(y_1))\}.
\end{eqnarray}
When $d^2f=0$ in the above equations \eqref{2-cocycle-f1}--\eqref{2-cocycle-f4}, $f$ is called a 2-cocycle.

\section{Abelian extensions}\label{ext}
In this section, we study abelian extensions of Hom-Lie antialgebra. It is proved that the equivalent classes of abelian extensions of Hom-Lie antialgebras are in one-to-one correspondence to the elements of the second cohomology group.
\begin{defi}
Let $(V,\a_{V_0},\b_{V_1})$, $(\fa,\a,\b)$ and $(\widetilde{\fa},\widetilde{\a},\widetilde{\b})$ be Hom-Lie antialgebras. An extension of $\fa$ by $V$ is a short exact sequence
\begin{equation*}
\xymatrix{
  0 \ar[r]^{} & V \ar[r]^{i} & \widetilde{\fa}  \ar[r]^{p} & \fa \ar[r]^{} & 0 }
\end{equation*}
of Hom-Lie antialgebras. It is  called an abelian extension, if $V$ is an abelian ideal of $\widetilde{\fa}$, i.e.
$V_0\cdot V_0=V_0\cdot V_1=[V_1, V_1]=0$ and $\widetilde{\fa}\cdot V\subseteq V$.
\end{defi}

\begin{defi}
 Two extensions of Hom-Lie antialgebra
 $$E_{\widetilde{\fa}}:\xymatrix{
  0 \ar[r] & V \ar[r]^{i} & \widetilde{\fa} \ar[r]^{p} & \fa \ar[r] & 0 }$$
  and
  $$E_{\widehat{\fa}} : \xymatrix{
  0 \ar[r] & V \ar[r]^{j} & \widehat{\fa} \ar[r]^{q} & \fa \ar[r] & 0 }$$
  are called equivalent, if there exists a Hom-Lie antialgebra homomorphism $\phi=(\phi_0,\phi_1):\widetilde{\fa}\rightarrow\widehat{\fa}$ such that $\phi\circ i=j$, $q\circ\phi=p$, that is to say the following diagram commutes
  \begin{equation*}
 \xymatrix{
   0 \ar[r]^{} & V \ar[d]_{id} \ar[r]^{i} & \widetilde{\fa} \ar[d]_{\phi} \ar[r]^{p} & \fa \ar[d]_{id} \ar[r]^{} & 0 \\
   0 \ar[r]^{} & V \ar[r]^{j} & \widehat{\fa} \ar[r]^{q} & \fa \ar[r]^{} & 0.   }
\end{equation*}
 We denote by $Ext(\fa,V)$  the set of equivalence classes of extensions of $\fa$ by $V$.
\end{defi}

A section $\sigma:\fa\rightarrow\widetilde{\fa}$ of $p:\widetilde{\fa}\rightarrow\fa$ consists of linear maps $\sigma_0:\fa_0\rightarrow\widetilde{\fa}_0$, $\sigma_1:\fa_1\rightarrow\widetilde{\fa}_1 $ such that $p_0\circ\sigma_0=id_{\fa_0}$, $p_1\circ\sigma_1=id_{\fa_1}$.
Define the following maps
\begin{eqnarray}
\r_0:\fa_0\rightarrow \End(V)_0,\quad\r_1:\fa_1\rightarrow \End(V)_1
\end{eqnarray}
by
\begin{eqnarray}
\label{repre1}
\r_0(x)(u +w )\triangleq\sigma_0(x)\cdot u +\sigma_0(x)\cdot w,\\
\label{repre2}
\r_1(y)(u +w )\triangleq\sigma_1 (y)\cdot u +[\sigma_1(y),w].
\end{eqnarray}
These two maps are well defined since $V$ is abelian. This gives a representation of $\fa$ on $V$ as the following lemma shows.
The proof is routine, so we omit the details.
\begin{lemma}
With above notations, $(\r_0,\r_1)$ is a representation of $\fa$. Moreover, equivalent abelian extensions lead to the same representation.
\end{lemma}

Let $\sigma=(\sigma_0,\sigma_1):\fa\rightarrow\widetilde{\fa}$ be a section of an abelian extension. Define the following maps:
\begin{eqnarray}\label{cocy1}
\omega_0(x_1,x_2)&\triangleq&\sigma_0(x_1)\cdot\sigma_0(x_2)-\sigma_0(x_1\cdot x_2)\in V_0,\\
\label{cocy2}
\omega_1(x_1,y_1)&\triangleq&\sigma_0(x_1)\cdot\sigma_1(y_1)-\sigma_1(x_1\cdot y_1)\in V_1,\\
\label{cocy3}
\omega_2(y_1,y_2)&\triangleq&[\sigma_1(y_1),\sigma_1(y_2)]-\sigma_0([y_1,y_2])\in V_0,
\end{eqnarray}
for all $x_i\in\fa_0$ and $y_i\in\fa_1$.
\begin{lemma}
Let $\xymatrix@C=0.5cm{
  0 \ar[r] & V \ar[r]^{} & \widetilde{\fa} \ar[r]^{} & \fa \ar[r] & 0 }$ be an abelian extension of $\fa$ by $V$. Then $(\omega_0,\omega_1,\omega_2)$ defined by \eqref{cocy1}--\eqref{cocy3} is a $2$-cocycle of $\fa$ with coefficients in $V$, where the representation $\r=(\r_0,\r_1)$ is given by \eqref{repre1}--\eqref{repre2}.
\end{lemma}
\begin{proof}
First, we prove that $\omega=(\omega_0,\omega_1,\omega_2)$ is a 2-cochain.
Since $\a$ and $\b$ form an algebraic homomorphism of $\fa$, we have
\begin{eqnarray*}
&&\alpha_{V_0}\omega_0(x_1,x_2)\\
&=&\alpha_{V_0}(\sigma_0(x_1)\cdot\sigma_0(x_2)-\sigma_0(x_1\cdot x_2))\\
&=&\alpha_{V_0}\sigma_0(x_1)\cdot\alpha_{V_0}\sigma_0(x_2)-\alpha_{V_0}\sigma_0(x_1\cdot x_2))\\
&=&\sigma_0\alpha(x_1)\cdot\sigma_0\alpha(x_2)-\sigma_0(\alpha (x_1)\cdot \alpha (x_2))\\
&=&\omega_0(\alpha(x_1),\alpha(x_2)).
\end{eqnarray*}
Thus we get
\begin{eqnarray*}
\alpha_{V_0}\omega_0(x_1,x_2)=\omega_0(\alpha(x_1),\alpha(x_2)).
\end{eqnarray*}
Similarly, one get
\begin{eqnarray*}
\alpha_{V_1}\omega_1(x_1,y_1)&=&\omega_1(\alpha(x_1),\alpha(y_1)),\\
\alpha_{V_0}\omega_2(y_1,y_2)&=&\omega_2(\alpha(y_1),\alpha(y_2)).
\end{eqnarray*}

Second, we prove that $\omega=(\omega_0,\omega_1,\omega_2)$ is a 2-cocycle.
By the equality
$$
\widetilde{\a}(\sigma_0(x_1))\cdot(\sigma_0(x_2)\cdot\sigma_0(x_3))
=(\sigma_0(x_1)\cdot\sigma_0(x_2))\cdot\widetilde{\a}(\sigma_0(x_3))
$$
and equation \eqref{cocy1}, we have
$$
\sigma_0(\a(x_1))\cdot\omega_0(x_2,x_3)+\omega_0(\a(x_1), x_2\cdot x_3))+\sigma_0(\a(x_1)\cdot(x_2\cdot x_3))
$$
$$
=\omega_0(x_1,x_2)\cdot\sigma_0(\a(x_3)+\omega_0(x_1\cdot x_2,\a(x_3))+\sigma_0((x_1\cdot x_2)\cdot\a(x_3)),
$$
Thus we obtain that
\begin{eqnarray}\label{cocycle1}
\nonumber&&\r_0(\a(x_1))\omega_0(x_2,x_3)+\omega_0(\a(x_1),x_2\cdot x_3)\\
&=&\r_0(\a(x_3))\omega_0(x_1,x_2)+\omega_0(x_1\cdot x_2,\a(x_3)).
\end{eqnarray}
Similarly, by the equality
$$
\widetilde{\a}(\sigma_0(x_1))\cdot(\sigma_0(x_2)\cdot\sigma_1(y_1))
=\half(\sigma_0(x_1)\cdot\sigma_0(x_2))\cdot\widetilde{\b}(\sigma_1(y_1)),
$$
we obtain that
\begin{eqnarray}\label{cocycle2}
\nonumber&&\r_0(\a(x_1))\omega_1(x_2,y_1)+\omega_1(\a(x_1),x_2\cdot y_1)\\
&=&\half\r_1(\b(y_1))\omega_0(x_1,x_2)+\half\omega_1(x_1\cdot x_2,\b(y_1)).
\end{eqnarray}
By the equality
\begin{eqnarray*}
\widetilde{\a}(\sigma_0(x_1))\cdot[\sigma_1(y_1),\sigma_1(y_2)]_{\widetilde{\fa}_1 }
&=&[\sigma_0(x_1)\cdot \sigma_1(y_1),\widetilde{\b}(\sigma_1(y_2))]\\
&&+[\widetilde{\b}(\sigma_1(y_1)),\sigma_0(x_1)\cdot\sigma_1(y_2)],
\end{eqnarray*}
we obtain that
\begin{eqnarray}\label{cocycle3}
 \nonumber
 &&\r_0(\a(x_1))\omega_2(y_1,y_2)+\omega_0(\a(x_1),[y_1,y_2])\\
&=&\r_1(\b(y_2))\omega_1(x_1,y_1)+\omega_2(x_1\cdot y_1,\b(y_2))\notag\\
 &&+\r_1(\b(y_1))\omega_1(x_1,y_2)+\omega_2(\b(y_1),x_1\cdot y_2).
\end{eqnarray}
By the equality
\begin{eqnarray}
\nonumber&&\widetilde{\b}(\sigma_1(y_1))\cdot[\sigma_1(y_2),\sigma_1(y_3)]_{\widetilde{\fa}_1 }+
\widetilde{\b}(\sigma_1(y_2))\cdot[\sigma_1(y_3),\sigma_1(y_1)]_{\widetilde{\fa}_1 }\\
\nonumber&&+\widetilde{\b}(\sigma_1(y_3))\cdot[\sigma_1(y_1),\sigma_1(y_2)]_{\widetilde{\fa}_1 }=0,
\end{eqnarray}
we obtain that
\begin{eqnarray}\label{cocycle4}
\nonumber&&\r_1(\b(y_1))\omega_2(y_2,y_3)+\omega_1(\b(y_1),[y_2,y_3])\\
\nonumber&&+\r_1(\b(y_2))\omega_2(y_3,y_1)+\omega_1(\b(y_2),[y_3,y_1])\\
&&+\r_1(\b(y_3))\omega_2(y_1,y_2)+\omega_1(\b(y_3),[y_1,y_2])=0.
\end{eqnarray}
When substituting  $\half\omega_0(x_1,x_2)$, $\omega_1(x_1,y_1)$ and $\omega_2(y_1,y_2)$ with  $f(x_1,x_2)$,  $f(x_1,y_1)$ and $f(y_1,y_2)$,
one can see that the equations \eqref{cocycle1}--\eqref{cocycle4} correspond to \eqref{2-cocycle-f1}--\eqref{2-cocycle-f4}, which are exactly the $2$-cocycle conditions given at the end of last section.
\end{proof}

Now, we can obtain a Hom-Lie antialgebra structure on the space $\fa\oplus V$ using the $2$-cocycle given above.
\begin{lemma}
Let $(\fa,\a,\b)$ be a Hom-Lie antialgebra, $(V,\r)$ be an $\fa$-module and $\omega=(\omega_0,\omega_1,\omega_2)$ is a $2$-cocycle. Then $(\fa\oplus V,\a+\a_{V_0},\b+\b_{V_1})$ is a Hom-Lie antialgebra under the following operations:
\begin{eqnarray}
\label{homo1}(\a+\a_{V_0})(x_1,u_1)&=&(\a(x_1),\a_{V_0}(u_1)),\notag\\
\label{homo2}(\b+\b_{V_1})(y_1,w_1)&=&(\b(y_1),\b_{V_1}(w_1)),\notag\\
\label{op1}
(x_1,u_1)\cdot(x_2,u_2)&=&(x_1\cdot x_2,~ \r_0(x_1)(u_2)+\r_0(x_2)(u_1)
+\omega_0(x_1,x_2)),\notag\\
\label{op2}
(x_1,u_1)\cdot(y_1,w_1)&=&(x_1\cdot y_1,~ \r_0(x_1)(w_1)+\r_1(y_1)(u_1)
+\omega_1(x_1,y_1)),\notag\\
\label{op3}
[(y_1,w_1),(y_2,w_2)]&=&([y_1,y_2], \r_1(y_1)(w_2)-\r_1(y_2)(w_1)
+\omega_2(y_1,y_2)),\notag
\end{eqnarray}
where $x_1,x_2\in\fa_0,~y_1,y_2\in\fa_1, u_1\in V_0, w_1\in V_1$.
\end{lemma}
\begin{proof} We are going to check that $\fa\oplus V$ with the above operations satisfies four axioms of Hom-Lie antialgebra.
By direct computations, we have
\begin{eqnarray*}
&&(\a+\a_{V_0})((x_1,u_1))\cdot((x_2,u_2)\cdot(x_3,u_3))\\
&=&(\a(x_1),\a_{V_0}(u_1))\cdot((x_2,u_2)\cdot(x_3,u_3))\\
&=&(\underbrace{\a(x_1)\cdot(x_2\cdot x_3)}_A,
~\underbrace{\r_0(\a(x_1))(\r_0(x_2)(u_3))}_B
+\underbrace{\r_0(\a(x_1))(\r_0(x_3)(u_2))}_C\\
&&+\underbrace{\r_0(x_2\cdot x_3)(\a_{V_0}(u_1))}_D
+\underbrace{\r_0(\a(x_1))\omega_0(x_2,x_3)+\omega_0(\a(x_1),x_2\cdot x_3)}_E)
\end{eqnarray*}
\begin{eqnarray*}
&&((x_1,u_1)\cdot(x_2,u_2))\cdot(\a+\a_{V_0})((x_3,u_3))\\
&=&((x_1,u_1)\cdot(x_2,u_2))\cdot(\a(x_3),\a_{V_0}(u_3))\\
&=&(\underbrace{(x_1\cdot x_2)\cdot \a(x_3)}_{A'},
~\underbrace{\r_0(x_1\cdot x_2)(\a_{V_0}(u_3))}_{B'}
+\underbrace{\r_0(\a(x_3))\r_0(x_1)(u_2)}_{C'}\\
&&+\underbrace{\r_0(\a(x_3))\r_0(x_2)(u_1)}_{D'}
+\underbrace{\r_0(\a(x_3))\omega_0(x_1,x_2)+\omega_0(x_1\cdot x_2,\a(x_3)}_{E'})
\end{eqnarray*}
Since $\fa$ is a Hom-Lie antialgebra and by the representation condition \eqref{rep01}, we have
$$A=A', ~~B=B', ~~C=C', ~~D=D'.$$
Due to the 2-cocycle condition \eqref{cocycle1}, we have
$$E=E'.$$
Thus we obtain
\begin{eqnarray}\notag
&&(\a+\a_{V_0})((x_1,u_1))\cdot((x_2,u_2)\cdot(x_3,u_3))\\
&=&((x_1,u_1)\cdot(x_2,u_2))\cdot(\a+\a_{V_0})((x_3,u_3)).\label{semi01}
\end{eqnarray}
Analogously, by using the representation condition \eqref{rep02} and the 2-cocycle condition \eqref{cocycle2},
one can prove the following equality:
\begin{eqnarray}\notag
&&(\a+\a_{V_0})((x_1,u_1))\cdot((x_2,u_2)\cdot(y_1,w_1))\\
&=&\half((x_1,u_1)\cdot(x_2,u_2))\cdot(\b+\b_{V_1})(y_1,w_1).\label{semi02}
\end{eqnarray}

By direct computations, we also have
\begin{eqnarray*}
&&(\a+\a_{V_0})(x_1,u_1)\cdot[(y_1,w_1)\cdot(y_2,w_2)]\\
&=&(\a(x_1),\a_{V_0}(u_1))\cdot[(y_1,w_1)\cdot(y_2,w_2)]\\
&=&(\underbrace{\a(x_1)\cdot[y_1,y_2]}_K,
~\underbrace{\r_0(\a(x_1))\r_1(y_1)(w_2)}_L
\underbrace{-\r_0(\a(x_1))\r_1(y_2)(w_1)}_{M}\\
&&+\underbrace{\r_0([y_1,y_2])(\a_{V_0}(u_1))}_{N}
+\underbrace{\r_0(\a(x_1))\omega_2(y_1,y_2)+\omega_0(\a(x_1),[y_1,y_2])}_{O})
\end{eqnarray*}
\begin{eqnarray*}
&&[(x_1,u_1)\cdot(y_1,w_1),(\b+\b_{V_1})(y_2,w_2)]\\
&=&[(x_1,u_1)\cdot(y_1,w_1)),(\b(y_2),\b_{V_1}(w_2))]\\
&=&\big(\underbrace{(x_1\cdot y_1)\cdot\b(y_2)}_{K_1},
~\underbrace{\r_1(x_1\cdot y_1)\cdot \b_{V_1}(w_2)}_{L_1}
\underbrace{-\r_1(\b(y_2))\r_0(x_1)(w_1)}_{M_1}\\
&&\underbrace{-\r_1(\b(y_2))\r_1(y_1)(u_1)}_{N_1}
\underbrace{-\r_1(\b(y_2))\omega_1(x_1,y_1)+\omega_2(x_1\cdot y_1,\b_{V_1}(w_2))}_{O_1}\big)
\end{eqnarray*}
\begin{eqnarray*}
&&[(\b+\b_{V_1})(y_1,w_1),(x_1,u_1)\cdot(y_2,w_2)]\\
&=&[(\b(y_1),\b_{V_1}(w_1)),(x_1,u_1)\cdot(y_2,w_2)]\\
&=&\big(\underbrace{\b(y_1)\cdot(x_1\cdot y_2)}_{K_2},
\underbrace{\r_1(\b(y_1))\r_0(x_1)(w_2)}_{L_2}
\underbrace{-\r_1(x_1\cdot y_2)\b_{V_1}(w_1)}_{M_2}\\
&&+\underbrace{\r_1(\b(y_1))\r_1(y_2)(u_1)}_{N_2}
+\underbrace{\r_1(\b(y_1))\omega_1(x_1,y_2)+\omega_2(x_1\cdot y_2,\b(y_1))}_{O_2}\big)
\end{eqnarray*}
Since $\fa$ is a Hom-Lie antialgebra and by the representation conditions \eqref{rep05} and \eqref{rep06}, we have
$$K=K_1+K_2,~L=L_1+L_2,~M=M_1+M_2,~N=N_1+N_2.$$
Due to the 2-cocycle condition \eqref{cocycle3}, we have
$$O=O_1+O_2.$$
Thus we get
\begin{eqnarray}\notag
&&(\a+\a_{V_0})(x_1,u_1)\cdot[(y_1,w_1), (y_2,w_2)]\\
&=&[(x_1,u_1)\cdot(y_1,w_1),(\b+\b_{V_1})(y_2,w_2)]+[(\b+\b_{V_1})(y_1,w_1),(x_1,u_1)\cdot(y_2,w_2)].\qquad\label{semi03}
\end{eqnarray}
Similarly, by the representation condition \eqref{rep07} and  the 2-cocycle condition \eqref{cocycle4}, we get
\begin{eqnarray}\notag
&&(\b+\b_{V_1})(y_1,w_1)\cdot[(y_2,w_2),(y_3,w_3)]\\
\notag &&+(\b+\b_{V_1})(y_2,w_2)\cdot[(y_3,w_3),(y_1,w_1)]\\
&&+(\b+\b_{V_1})(y_3,w_3)\cdot[(y_1,w_1),(y_2,w_2)]=0.\label{semi04}
\end{eqnarray}
Therefore, from equalities \eqref{semi01}--\eqref{semi04}, we obtain that $\fa\oplus V$ is a Hom-Lie antialgebra.
The proof is completed.
\end{proof}

\begin{lemma}
Two abelian extensions of Hom-Lie antialgebras
\begin{eqnarray*}
\xymatrix@C=0.5cm{
  0 \ar[r] & V \ar[r]^{} & \fa\oplus_{\omega} V \ar[r]^{} & \fa \ar[r] & 0 }
\end{eqnarray*}
and
\begin{eqnarray*}
\xymatrix@C=0.5cm{
  0 \ar[r] & V \ar[r]^{} & \fa\oplus_{\omega'} V \ar[r]^{} & \fa \ar[r] & 0 }
\end{eqnarray*}
are equivalent if and only if $\omega$ and $\omega'$ are in the same cohomology class.
\end{lemma}

\begin{proof} First, assume the above two abelian extensions are equivalent and $\phi=(\phi_0,\phi_1):\fa\oplus_{\omega} V\rightarrow \fa\oplus_{\omega'} V$ be the corresponding homomorphism.
As $\phi$ has to be the identity on $\fa$, then there must exists a map $f:\fa\rightarrow V$ such that
\begin{eqnarray}\label{equ1}
\phi_0(x_i,u_i)&=&(x_i,f_0(x_i)+u_i),\\
\phi_1(y_i,w_i)&=&(y_i,f_1(y_i)+w_i),\label{equ2}
\end{eqnarray}
where $f_0:\fa_0\to V_0$ and $f_1:\fa_1\to V_1$.

Since  $\phi$ is a homomorphism between Hom-Lie antialgebras $\fa\oplus_{\omega} V$ and $\fa\oplus_{\omega'} V$,  we have
\begin{eqnarray}\label{homo1}
\phi_0\big((x_1,u_1)\cdot(x_2,u_2)\big)&=&\phi_0(x_1,u_1)\cdot\phi_0(x_2,u_2),\\
\label{homo2}
\phi_1\big((x_1,u_1)\cdot(y_1,w_1)\big)&=&\phi_0(x_1,u_1)\cdot\phi_1(y_1,w_1),\\
\label{homo3}
\phi_0\big([(y_1,w_1),(y_2,w_2)]\big)&=&[\phi_1(y_1,w_1),\phi_1(y_2,w_2)].
\end{eqnarray}

The left hand side of \eqref{homo1} is equal to
\begin{eqnarray*}
 &&\phi_0(x_1\cdot x_2,~\r_0(x_1)(u_2)+\r_0(x_2)(u_1)+\omega_0(x_1,x_2))\\
&=&(x_1\cdot x_2,~f_0(x_1\cdot x_2)+\r_0(x_1)(u_2)+\r_0(x_2)(u_1)+\omega_0(x_1,x_2))
\end{eqnarray*}
and the right hand side of \eqref{homo1} is equal to
\begin{eqnarray*}
&&(x_1,f_0(x_1)+u_1)\cdot(x_2,f_0(x_2)+u_2)\\
&=&(x_1\cdot x_2,~\r_0(x_1)f_0(x_2)+\r_0(x_1)(u_2)\\
&&+\r_0(x_2)f_0(x_1)+\r_0(x_2)(u_1)+\omega_0'(x_1,x_2)).
\end{eqnarray*}
Thus we obtain
\begin{eqnarray}\label{bound1}
(\omega_0-\omega_0')(x_1,x_2)=\r_0(x_1)f_0(x_2)+\r_0(x_2)f_0(x_1)-f_0(x_1\cdot x_2).
\end{eqnarray}

By similar computations, we also obtain
\begin{eqnarray}\label{bound2}
(\omega_1-\omega_1')(x_1,y_1)&=&\r_0(x_1)f_1(y_1)+\r_1(y_1)f_0(x_1)-f_1(x_1\cdot y_1),\\
\label{bound3}(\omega_2-\omega_2')(y_1,y_2)&=&\r_1(y_1)f_1(y_2)-\r_1(y_2)f_1(y_1)-f_0([y_1,y_2]).
\end{eqnarray}
From equations \eqref{bound1}--\eqref{bound3}, we obtain that $\omega$ and $\omega'$ are in the same cohomology class.

Conversely, if $\omega$ and $\omega'$ are in the same cohomology class, there exists a coboundary map $f:\fa\to V$ such that $\omega-\omega'=df$. Then we can define the maps $\phi$ by \eqref{equ1} and \eqref{equ2}.
Similar as the above calculations, one can show that $\phi$ is an equivalence of the two abelian extensions. We omit the details.
This finished the proof.
\end{proof}

From the above lemmas, we obtain the main theorem of this section.
\begin{theo}\label{thm001}
Let $(\fa,\a,\b)$ be a Hom-Lie antialgebra and $(\r_0,\r_1)$ be a representation of  $(\fa,\a,\b)$ over $(V,\a_{V_0},\b_{V_1})$. Then there is a one-to-one correspondence between the set of equivalent classes of abelian extensions of the Hom-Lie antialgebra $(\fa,\a,\b)$ by $(V,\a_{V_0},\b_{V_1})$ and the elements in the second cohomology group $H^2(\fa,V)$.
\end{theo}

\section{Deformations}\label{def}
In this section, we study infinitesimal deformations of Hom-Lie antialgebras. The notion of Nijenhuis operators for Hom-Lie antialgebras is introduced. This kind of operators gives trivial deformation.

Let $(\fa,\alpha,\beta)$ be a Hom-Lie antialgebra and $\omega_0:\fa_0\times\fa_0\rightarrow\fa_0$, $\omega_1:\fa_0\times\fa_1\rightarrow\fa_1,\omega_2:\fa_1\times\fa_1\rightarrow\fa_0$ be bilinear maps. Consider a $t$-parametrized family of bilinear maps:
\begin{eqnarray*}
{x_1\cdot_t x_2}&=&x_1\cdot x_2+t\omega_0(x_1,x_2),\\
{x_1\cdot_t y_1}&=&x_1\cdot y_1+t\omega_1(x_1,y_1),\\
{[y_1,y_2]_t}&=&[y_1,y_2]+t\omega_2(y_1,y_2).
\end{eqnarray*}

If these maps endow $(\fa,\alpha,\beta)$ with a Hom-Lie antialgebra structure which is denoted by $(\fa_t,\alpha,\beta)$, then we say that $\omega=(\omega_0,\omega_1,\omega_2)$ generates a $t$-parameter infinitesimal deformation of Hom-Lie antialgebra $(\fa,\alpha,\beta)$.
\begin{theo}\label{TwoCon}
With the above notations, $\omega=(\omega_0,\omega_1,\omega_2)$ generates a $t$-parameter infinitesimal deformation of a Hom-Lie antialgebra $\fa$ if and only if the following two conditions hold:

(i) $(\omega_0,\omega_1,\omega_2)$ defines a Hom-Lie antialgebra structure on $\fa$;

(ii) $(\omega_0,\omega_1,\omega_2)$ is a 2-cocycle of $\fa$ with the coefficients in the adjoint representation.
\end{theo}
\begin{proof}
First, if we assume that $(\fa,\alpha,\beta)$ and $(\fa_t,\alpha,\beta)$ are multiplicative Hom-Lie antialgebras, then we have
\begin{eqnarray*}
\alpha(x_1\cdot_t x_2)&=&\alpha(x_1)\cdot_t\alpha(x_2),
\end{eqnarray*}
the left hand side is equal to
\begin{eqnarray*}
&&\alpha(x_1\cdot_t x_2)\\
&=&\alpha(x_1\cdot x_2+t\omega_0(x_1, x_2))\\
&=&\alpha(x_1)\cdot \alpha(x_2)+t\alpha\omega_0(x_1, x_2));
\end{eqnarray*}
the right hand side is equal to
\begin{eqnarray*}
&&\alpha(x_1)\cdot_t\alpha(x_2)\\
&=&\alpha(x_1)\cdot \alpha(x_2)+t\omega_0(\alpha(x_1), \alpha(x_2)).
\end{eqnarray*}
Thus we obtain
\begin{eqnarray}\label{mm1}
\alpha\omega_0(x_1, x_2))=\omega_0(\alpha(x_1), \alpha(x_2)).
\end{eqnarray}
Similarly, one obtain
\begin{eqnarray}
\label{mm2}\beta\omega_1(x_1, y_1))&=&\omega_0(\alpha(x_1), \beta(y_1))\\
\label{mm3}\alpha\omega_2(y_1, y_2))&=&\omega_0(\beta(y_1), \beta(y_2)).
\end{eqnarray}

Second, assume $(\omega_0,\omega_1,\omega_2)$ generates a $t$-parameter infinitesimal deformation of the Hom-Lie antialgebra $\fa$, then the maps $[\cdot,\cdot]_t$ defined above must satisfy \eqref{hanti01}--\eqref{hanti04}.

For the equality
\begin{eqnarray*}
\a(x_1)\cdot_t(x_2\cdot_t x_3)=(x_1\cdot_t x_2)\cdot_t\a(x_3),
\end{eqnarray*}
the left hand side is equal to
\begin{eqnarray*}
&&\a(x_1)\cdot_t(x_2\cdot x_3+t\omega_0(x_2,x_3))\\
&=&\a(x_1)\cdot(x_2\cdot x_3)+\a(x_1)\cdot t\omega_0(x_2,x_3)\\
&&+t\omega_0(\a(x_1),x_2\cdot x_3)+t\omega_0(\a(x_1),t\omega_0(x_2,x_3));
\end{eqnarray*}
the right hand side is equal to
\begin{eqnarray*}
&&(x_1\cdot x_2+t\omega_0(x_1,x_2))\cdot_t\a(x_3)\\
&=&(x_1\cdot x_3)\cdot\a(x_3)+t\omega_0(x_1,x_2)\cdot\a(x_3)\\
&&+t\omega_0(x_1\cdot x_2,\a(x_3))+t\omega_0(t\omega_0(x_1,x_2),\a(x_3)).
\end{eqnarray*}
Thus we have
\begin{eqnarray}\label{decocy1}
\nonumber&&\a(x_1)\cdot\omega_0(x_2,x_3)+\omega_0(\a(x_1),x_2\cdot x_3)\\
&=&\omega_0(x_1,x_2)\cdot\a(x_3)+\omega_0(x_1\cdot x_2,\a(x_3)),
\end{eqnarray}
and
\begin{eqnarray}\label{new1}
\omega_0(\a(x_1),\omega_0(x_2,x_3))=\omega_0(\omega_0(x_1,x_2),\a(x_3)).
\end{eqnarray}
For the equality
\begin{eqnarray*}
\a(x_1)\cdot_t(x_2\cdot_t y_1)=\half(x_1\cdot_t x_2)\cdot_t\b(y_1),
\end{eqnarray*}
the left hand side is equal to
\begin{eqnarray*}
&&\a(x_1)\cdot_t(x_2\cdot y_1+t\omega_1(x_1,y))\\
&=&\a(x_1)\cdot(x_2\cdot y_1)+\a(x_1)\cdot t\omega_1(x_1,y_1)\\
&&+t\omega_1(\a(x_1),x_2\cdot y_1)+t\omega_1(\a(x_1),t\omega_1(x_2,y_1));
\end{eqnarray*}
the right hand side is equal to
\begin{eqnarray*}
&&\half(x_1\cdot x_2+t\omega_0(x_1,x_2))\cdot_t\b(y_1)\\
&=&\half\big((x_1,x_2)\cdot\b(y_1)+t\omega_0(x_1,x_2)\cdot\b(y_1)\\
&&+t\omega_1(x_1\cdot x_2,\b(y_1))+t\omega_1(t\omega_0(x_1,x_2),\b(y_1))\big).
\end{eqnarray*}
Thus we have
\begin{eqnarray}\label{decocy2}
\nonumber&&\a(x_1)\cdot\omega_1(x_2,y_1)+\omega_1(\a(x_1),x_2\cdot y_1)\\
&=&\half\omega_0(x_1,x_2)\cdot\b(y_1)+\half\omega_1(x_1\cdot x_2,\b(y_1))
\end{eqnarray}
and
\begin{eqnarray}\label{new2}
\omega_1\big(\a(x_1),\omega_1(x_2,y_1)\big)=\half\omega_1(\omega_0(x_1,x_2),\b(y_1)).
\end{eqnarray}
For the equality
\begin{eqnarray*}
\a(x_1)\cdot_t[y_1, y_2]_t=[(x_1\cdot_t y_1),\b(y_2)]_t+[\b(y_1),(x_1\cdot_t y_2)]_t
\end{eqnarray*}
the left hand side is equal to
\begin{eqnarray*}
&&\a(x_1)\cdot_t\big([y_1,y_2]+t\omega_2(y_1,y_2)\big)\\
&=&\a(x_1)\cdot[y_1,y_2]+t\omega_0(\a(x_1),[y_1,y_2])\\
&&+\a(x_1)\cdot t\omega_2(y_1,y_2)+t\omega_0(\a(x_1),t\omega_2(y_1,y_2));
\end{eqnarray*}
the right hand side is equal to
\begin{eqnarray*}
&&[(x_1\cdot y_1+t\omega_1(x_1,y_1)),\b(y_2)]_t+[\b(y_1),(x\cdot y_2+t\omega_1(x_1,y_2))]_t\\
&=&[(x_1\cdot y_1),\b(y_2)]+t\omega_1(x_1,y_1)\cdot\b(y_2)\\
&&+t\omega_2(x_1\cdot y_1,\b(y_2))+t\omega_2(t\omega_1(x_1,y_1),\b(y_2))\\
&&+[\b(y_1),(x_1\cdot y_2)]+\b(y_1)\cdot t\omega_1(x_1,y_2)\\
&&+t\omega_2(\b(y_1),x\cdot y_2)+t\omega_2(\b(y_1),t\omega_1(x_1,y_2)).
\end{eqnarray*}
Thus we have
\begin{eqnarray}\label{decocy3}
\nonumber&&\omega_0(\a(x_1),[y_1,y_2])+\a(x_1)\cdot\omega_2(y_1,y_2)\\
\nonumber&=&\omega_1(x_1,y_1)\cdot\b(y_2)+\omega_2(x_1\cdot y_1,\b(y_2))\\
&&+\b(y_1)\cdot\omega_1(x_1,y_2)+\omega_2(\b(y_1),x_1\cdot y_2)
\end{eqnarray}
and
\begin{eqnarray}\label{new3}
\omega_0(\a(x_1),\omega_2(y_1,y_2))=\omega_2(\omega_1(x_1,y_1),\b(y_2))+\omega_2(\b(y_1),\omega_1(x_1,y_2)).
\end{eqnarray}
For the equality
\begin{eqnarray*}
\b(y_1)\cdot_t[y_2,y_3]_t+\b(y_2)\cdot_t[y_3,y_1]_t+\b(y_3)\cdot_t[y_1,y_2]_t=0
\end{eqnarray*}
we have
\begin{eqnarray}\label{decocy4}
\nonumber\b(y_1)\cdot\omega_2(y_2,y_3)+\omega_1(\b(y_1),[y_2,y_3])+\b(y_2)\cdot t\omega_2(y_3,y_1)\\
+\omega_1(\b(y_2),[y_3,y_1])+\b(y_3)\cdot\omega_2(y_1,y_2)+\omega_1(\b(y_3),[y_1,y_2])=0
\end{eqnarray}
and
\begin{eqnarray}\label{new4}
\omega_1(\b(y_1),\omega_2(y_2,y_3))+\omega_1(\b(y_2),\omega_2(y_3,y_1))+\omega_1(\b(y_3),\omega_2(y_1,y_2))=0.
\end{eqnarray}
Therefore, by \eqref{mm1}--\eqref{mm3}, \eqref{new1},\eqref{new2},\eqref{new3} and \eqref{new4}, $(\omega_0,\omega_1,\omega_2)$ defines a multiplicative Hom-Lie antialgebra on $\fa$. Futhermore, by \eqref{mm1}--\eqref{mm3}, \eqref{decocy1},\eqref{decocy2},\eqref{decocy3} and \eqref{decocy4}, we obtain that $(\omega_0,\omega_1,\omega_2)$ is a 2-cocycle of $\fa$ with coefficients in the adjoint representation.

Conversely, if $(\omega_0,\omega_1,\omega_2)$ defines a multiplicative Hom-Lie antialgebra on $\fa$ and it is a 2-cocycle of $\fa$ with coefficients in the adjoint representation, then one can check by the same reasoning as above that $\omega=(\omega_0,\omega_1,\omega_2)$ generates a $t$-parameter infinitesimal deformation of Hom-Lie antialgebra $(\fa,\alpha,\beta)$. This finished the proof.
\end{proof}

A deformation is said to be trivial if there exists a linear map $\phi=(\phi_0,\phi_1):\fa\rightarrow\fa$ such that for $\phi_t=(\phi_{0t},\phi_{1t}):\fa_t\rightarrow\fa$, $\phi_{0t}=id_{\fa_0}+t\phi_0,~\phi_{1t}=id_{\fa_1}+t\phi_1$ the following holds:
\begin{eqnarray}
\phi_{0t}\circ\a=\a\circ\phi_{0t},~~\phi_{1t}\circ\b=\b\circ\phi_{1t},\label{dehomo2}
\end{eqnarray}
and
\begin{eqnarray}
\nonumber\phi_{0t}(x_1\cdot_t x_2)&=&\phi_{0t}(x_1)\cdot\phi_{0t}(x_2),\\
\phi_{1t}(x_1\cdot_t y_1)&=&\phi_{0t}(x_1)\cdot\phi_{1t}(y_1),\\\label{dehomo1}
\nonumber\phi_{0t}[y_1,y_2]_t&=&[\phi_{1t}(y_1),\phi_{1t}(y_2)].
\end{eqnarray}
That is to say  $\phi_t$ is  a Hom-Lie antialgebra homomorphism from $\fa_t$ to $\fa$.

From equation \eqref{dehomo2}, we get
\begin{eqnarray}
\phi_{0}\circ\a=\a\circ\phi_{0},~~\phi_{1}\circ\b=\b\circ\phi_{1}.\label{Nij00}
\end{eqnarray}

From equation \eqref{dehomo1}, we have
\begin{eqnarray*}
\phi_{0t}(x_1\cdot_t x_2)&=&(id_{\fa_0}+t\phi_0)(x_1\cdot x_2+t\omega_0(x_1,x_2))\\
&=&x_1\cdot x_2+t\big(\omega_0(x_1,x_2)+\phi_0(x_1\cdot x_2)\big)+t^2\phi_0\omega_0(x_1,x_2),
\end{eqnarray*}
and
\begin{eqnarray*}
\phi_{0t}(x_1)\cdot\phi_{0t}(x_2)&=&(x_1+t\phi_0(x_1))\cdot(x_2+t\phi_0(x_2))\\
&=&x_1\cdot x_2+t(x_1\cdot\phi_0(x_2)+\phi_0(x_1)\cdot x_2)+t^2\phi_0(x_1)\cdot\phi_0(x_2).
\end{eqnarray*}
Thus, we get
\begin{eqnarray}
\omega_0(x_1,x_2)&=&x_1\cdot\phi_0(x_2)+\phi_0(x_1)\cdot x_2-\phi_0(x_1\cdot x_2),\\
\phi_0\omega_0(x_1,x_2)&=&\phi_0(x_1)\cdot\phi_0(x_2),
\end{eqnarray}
and
\begin{eqnarray}
\phi_0(x_1\cdot\phi_0(x_2))+\phi_0(\phi_0(x_1)\cdot x_2)-\phi_0^2(x_1\cdot x_2)=\phi_0(x_1)\cdot\phi_0(x_2).\label{Nij01}
\end{eqnarray}
Similarly,
\begin{eqnarray*}
\phi_{1t}(x\cdot_t y)&=&(id_{\fa_1}+t\phi_1)(x\cdot y+t\omega_1(x,y))\\
&=&x\cdot y+t\omega_1(x,y)+t\phi_1(x\cdot y)+t^2\phi_1\omega_1(x,y),
\end{eqnarray*}
\begin{eqnarray*}
\phi_{0t}(x)\cdot\phi_{1t}(y)&=&(x+t\phi_0(x))\cdot(y+t\phi_1(y))\\
&=&x\cdot y+x\cdot t\phi_1(y)+t\phi_0(x)\cdot y+t^2\phi_0(x)\cdot\phi_1(y),
\end{eqnarray*}
we get
\begin{eqnarray}
\omega_1(x,y)&=&x\cdot\phi_1(y)+\phi_0(x)\cdot y-\phi_1(x\cdot y),\\
\phi_1\omega_1(x,y)&=&\phi_0(x)\cdot\phi_1(y),
\end{eqnarray}
and
\begin{eqnarray}
\phi_1(x\cdot\phi_1(y))+\phi_1(\phi_0(x)\cdot y)-\phi_1^2(x\cdot y)=\phi_0(x)\cdot\phi_1(y).\label{Nij02}
\end{eqnarray}
Next, due to
\begin{eqnarray*}
\phi_{0t}[y_1,y_2]_t&=&(id_{\fa_1}+t\phi_0)([y_1,y_2]+t\omega_2(y_1,y_2))\\
&=&[y_1,y_2]+t\omega_2(y_1,y_2)+t\phi_0([y_1,y_2])+t^2\phi_0\omega_2(y_1,y_2)
\end{eqnarray*}
and
\begin{eqnarray*}
[\phi_{1t}(y_1),\phi_{1t}(y_2)]&=&[y+t\phi_1(y_1),y_2+t\phi_1(y_2)]\\
&=&[y_1,y_2]+[y_1,t\phi_1(y_2)]+[t\phi_1(y_1),y_2]+[t\phi_1(y_1),t\phi_1(y_2)],
\end{eqnarray*}
we get
\begin{eqnarray}
\omega_2(y_1,y_2)&=&[y_1,\phi_1(y_2)]+[\phi_1(y_1),y_2]-\phi_0([y_1,y_2]),\\
\phi_0\omega_2(y_1,y_2)&=&[\phi_1(y_1),\phi_1(y_2)],
\end{eqnarray}
and
\begin{eqnarray}
\phi_0[y_1,\phi_1(y_2)]+\phi_0[\phi_1(y_1),y_2]-\phi_0^2([y_1,y_2])=[\phi_1(y_1),\phi_1(y_2)].\label{Nij03}
\end{eqnarray}

\begin{defi}
A linear operator $\phi:\fa\rightarrow\fa$ is called a Nijenhuis operator if and only if\eqref{Nij00},  \eqref{Nij01}, \eqref{Nij02} and \eqref{Nij03} hold.
\end{defi}
We have seen that every trivial deformation produces a Nijenhuis operator. Conversely, Nijenhuis operator gives a trivial deformation as the following theorem shows.
\begin{theo}
Let $\phi$ be a Nijenhuis operator on $\fa$. Then a deformation of $\fa$ can be obtained by putting
\begin{eqnarray}
\omega_0(x_1,x_2)&=&x_1\cdot\phi_0(x_2)+\phi_0(x_1)\cdot x_2-\phi_0(x_1\cdot x_2),\\
\omega_1(x_1,y_1)&=&x_1\cdot\phi_1(y_1)+\phi_0(x_1)\cdot y_1-\phi_1(x_1\cdot y_1),\\
\omega_2(y_1,y_2)&=&[y_1,\phi_1(y_2)]+[\phi_1(y_1),y_2]-\phi_0([y_1,y_2]).
\end{eqnarray}
Furthermore, this deformation is a trivial one.
\end{theo}
\begin{proof}
By definition, $(\omega_0,\omega_1,\omega_2)$ is a 2-cocycle of $\fa$ with coefficients in the adjoint representation. In the following, we will show that $(\fa,\omega,\a,\b)$ is a Hom-Lie antialgebra of deformation type.

First we check the Hom-associative equality \eqref{hanti01} hold for $\omega_0$. A direct computations shows that
\begin{eqnarray*}
&&\omega_0(\omega_0(x_1,x_2),\a(x_3))\\
&=&\omega_0(x_1,x_2)\cdot\phi_0(\a(x_3))+\phi_0(\omega_0(x_1,x_2))\cdot\a(x_3)-\phi_0(\omega_0(x_1,x_2)\cdot\a(x_3))\\
&=&\underline{(x_1\cdot\phi_0(x_2))\cdot\phi_0(\a(x_3))}+\underline{(\phi_0(x_1)\cdot x_2)\cdot\phi_0(\a(x_3))}-\phi_0(x_1\cdot x_2)\cdot\phi_0(\a(x_3))\\
&&+\underline{(\phi_0(x_1)\cdot\phi_0(x_2))\cdot\a(x_3)}\\
&&-\phi_0((x_1\cdot\phi_0(x_2))\cdot\a(x_3))-\underline{\phi_0((\phi_0(x_1)\cdot x_2)\cdot\a(x_3))}+\phi_0(\phi_0(x_1\cdot x_2)\cdot\a(x_3)),
\end{eqnarray*}
and
\begin{eqnarray*}
&&\omega_0(\a(x_1),\omega_0(x_2,x_3))\\
&=&\a(x_1)\cdot\phi_0(\omega_0(x_2,x_3))+\phi_0(\a(x_1))\cdot\omega_0(x_2,x_3)-\phi_0(\a(x_1)\cdot\omega_0(x_2,x_3))\\
&=&\underline{\a(x_1)\cdot(\phi_0(x_2)\cdot\phi_0(x_3))}\\
&&+\underline{\phi_0(\a(x_1))\cdot(x_2\cdot\phi_0(x_3))}+\underline{\phi_0(\a(x_1))\cdot(\phi_0(x_2)\cdot x_3)}
-\phi_0(\a(x_1))\cdot\phi_0(x_2\cdot x_3)\\
&&-\phi_0(\a(x_1)\cdot(x_2\cdot\phi_0(x_3)))-\underline{\phi_0(\a(x_1)\cdot(\phi_0(x_2)\cdot x_3))}+\phi_0(\a(x_1)\cdot\phi_0(x_2\cdot x_3)).
\end{eqnarray*}
Since the Hom-associative identity \eqref{hanti01} hold for $(\fa,\a,\b)$, the underline items are cancelled.
The remaining items are cancelled by using the Nijenhuis operator condition \eqref{Nij01} of the forms:
\begin{eqnarray*}
\phi_0(x_1\cdot x_2)\cdot \phi(\a(x_3))&=&\phi_0((x_1\cdot x_2) \phi_0(\a(x_3))+\phi_0(x_1\cdot x_2) \cdot \a(x_3)\\
&&-\phi_0((x_1\cdot x_2)\cdot \a(x_3))),\\
\phi_0(\a(x_1))\cdot \phi(x_2\cdot x_3)&=&\phi_0(\a(x_1) \cdot\phi_0(x_2\cdot x_3)+\phi_0(\a(x_1))\cdot (x_2 \cdot x_3)\\
&&-\phi_0(\a(x_1)\cdot(x_2\cdot x_3)).
\end{eqnarray*}
Thus we get
\begin{eqnarray}\label{NHL1}
\omega_0(\a(x_1),\omega_0(x_2,x_3))=\omega_0(\omega_0(x_1,x_2),\a(x_3)).
\end{eqnarray}

Second, by similar computations using the Nijenhuis operator conditions \eqref{Nij02} and \eqref{Nij03} we get
\begin{eqnarray}\label{NHL2}
&&\omega_1(\a(x_1),\omega_1(x_2,y_1))=\half\omega_1(\omega_0(x_1,x_2),\b(y_1)),\\
\label{NHL3}
&&\omega_0(\a(x),\omega_2(y_1,y_2))=\omega_2(\omega_1(x,y_1),\b(y_2))+\omega_2(\b(y_1),\omega_1(x,y_2)),\\
\label{NHL4}
&&\omega_1(\b(y_1),\omega_2(y_2,y_3))+\omega_1(\b(y_2),\omega_2(y_3,y_1))+\omega_1(\b(y_3),\omega_2(y_1,y_2))=0.
\end{eqnarray}
Thus the equality \eqref{hanti02}--\eqref{hanti04} hold for $\omega$.
From the above calculations, we obtain that $(\fa,\omega,\a,\b)$ is a Hom-Lie antialgebra. Therefore, $(\omega_0,\omega_1,\omega_2)$ satisfies two conditions in Theorem \ref{TwoCon} and it gives a trivial deformation.
\end{proof}

\subsection*{Acknowledgements}
The author would like to thank the  referee for careful reading of the manuscript and for valuable suggestions which helped us both in English and in depth to improve the quality of the paper.
This research was supported by NSFC(11501179, 11961049) and a doctoral research program of
Henan Normal University.

\end{document}